\documentclass[12pt, reqno]{amsart}
\usepackage{amsmath, amsthm, amscd, amsfonts, amssymb, graphicx, color}
\usepackage[bookmarksnumbered, colorlinks, plainpages]{hyperref}
\hypersetup{colorlinks=true,linkcolor=red, anchorcolor=green, citecolor=cyan, urlcolor=red, filecolor=magenta, pdftoolbar=true}

\newtheorem{theorem}{Theorem}[section]
\newtheorem{lemma}[theorem]{Lemma}
\newtheorem{proposition}[theorem]{Proposition}
\newtheorem{cor}[theorem]{Corollary}
\newtheorem{definition}[theorem]{Definition}
\newtheorem{example}[theorem]{Example}

\theoremstyle{remark}
\newtheorem{remark}[theorem]{\bf{Remark}}
\numberwithin{equation}{section}
\allowdisplaybreaks
\begin{document}

\title [Numerical range of Toeplitz and Composition operators]{\Small{Numerical range of Toeplitz and Composition operators on weighted Bergman spaces}}
	\author[A. Sen, S. Halder, R. Birbonshi, K. Paul]{Anirban Sen, Subhadip Halder, Riddhick Birbonshi, Kallol Paul}
	
\address[Sen] {Department of Mathematics, Jadavpur University, Kolkata 700032, West Bengal, India}
\email{anirbansenfulia@gmail.com}

\address[Halder]{Department of Mathematics, Jadavpur University, Kolkata 700032, West Bengal, India}
\email{subhadiphalderju@gmail.com}

\address[Birbonshi] {Department of Mathematics, Jadavpur University, Kolkata 700032, West Bengal, India}
\email{riddhick.math@gmail.com}

\address[Paul] {Department of Mathematics, Jadavpur University, Kolkata 700032, West Bengal, India}
\email{kalloldada@gmail.com}




\maketitle	

\begin{abstract} 
	In this paper we completely describe the numerical range of Toeplitz operators on weighted Bergman spaces with harmonic symbol. We also characterize the numerical range of weighted composition operators on weighted Bergman spaces and classify some sets which are the numerical range of composition operators. We investigate the inclusion of zero in the numerical range, and compute the radius of circle and ellipse contained in the numerical range of weighted composition operators on weighted Bergman spaces.
\end{abstract}

\section{Introduction}
Let $\mathbb{B}(\mathbb{H})$ be the $C^*$-algebra of all bounded linear operators on a complex Hilbert space $\mathbb{H}.$ For $T \in \mathbb{B}(\mathbb{H}),$ the numerical range of $T,$ denoted by $W(T),$ is the subset of the complex plane $\mathbb{C}$ defined by
$$W(T)=\left\{\langle Tf,f \rangle : f \in \mathbb{H}, \|f\|=1 \right\}.$$ 
 It is well known that $W(T)$ is a bounded and convex subset of $\mathbb{C}$. The spectrum of $T,$ denoted by $\sigma(T),$ is contained in the closure of $W(T).$ We refer to \cite{GR_BOOK,wu_gau_book} for the detailed proofs of these results and other properties of the numerical range.

The numerical range of Toeplitz and composition operators have been studied over the years. In 1972, Klein \cite{K_PAMS_1972} completely described the numerical range of Toeplitz operators on the Hardy-Hilbert space of the unit disk. Later, the numerical range of Toeplitz operators on Bergman space and polydisk were studied in \cite{G_AAA_2009,T_JOT_1995,WW_IEOT_2009}.
In \cite{M_LAA_2001}, the numerical ranges of composition operators on the Hardy-Hilbert space induced by monomials were characterized. Bourdon and Shapiro \cite{BS_IEOT_2002,BS_JMAA_2000} studied the numerical range of composition operators and the containment of the origin. 
Recently, the numerical range of weighted composition operators on Hardy-Hilbert space and  weighted Bergman spaces were explored in \cite{GJS_JMAA_2014,ZOR_JMAA_2018}.

In this article, we study the numerical ranges of Toeplitz operators and weighted composition operators on weighted Bergman spaces. The article is structured as follows. In Section \ref{S0}, we introduce some notations, recall some definitions, and present some preliminary results. In Section \ref{S1}, we completely describe the numerical range of Toeplitz operators on the weighted Bergman spaces with harmonic symbol. We provide an example to justify that the harmonic condition is necessary for the characterization given in Theorem \ref{T3}. Then we obtain the numerical range of some particular classes of weighted composition operators on the weighted Bergman spaces. Further, we characterize some bounded and convex sets which are the numerical range of weighted composition operators. Our main aim of Section \ref{S2} is to study when the origin is contained in the numerical range of weighted composition operators acting on weighted Bergman spaces. Then we derive some sufficient conditions about the closedness of the numerical range of weighted composition operators. In Section \ref{S3}, we identify several classes of weighted composition operators whose numerical range includes a circle or an ellipse. Furthermore, we determine the radius of the circle as well as the lengths of the minor and major axes of the ellipse.

\section{Preliminaries}\label{S0}

Let $\mathbb{D}=\{z \in \mathbb{C} : |z|<1\}$ be the open unit disk. Let $H(\mathbb{D})$ be the space all analytic functions on $\mathbb{D}$ and $H^{\infty}$ be the space of all bounded analytic functions on $\mathbb{D}.$ We will use the notations $\overline{X}, \partial X,int~X, Rel~int~X$ and $X^{\wedge}$ for the closure, boundary, interior, relative interior and convex hull, respectively, of the set $X \subset \mathbb{C}.$ 

For $\alpha>-1,$ the weighted Bergman space $L_a^2(dA_{\alpha})$ on the unit disk is defined as
$$L_a^2(dA_{\alpha})=\left\{f\in H(\mathbb{D}) : \int_{\mathbb{D}}|f(z)|^2dA_{\alpha}(z)<\infty \right\},$$
where $dA$ denotes the normalized Lebesgue area measure on $\mathbb{D}$ and 
$$dA_{\alpha}(z)=(\alpha+1)(1-|z|^2)^{\alpha}dA(z).$$
Clearly, $L_a^2(dA_{\alpha})=H(\mathbb{D})\cap L^2(\mathbb{D}, dA_{\alpha})$ and when the weight parameter $\alpha=0,$ the weighted Bergman space becomes the classical Bergman space $L_a^2(dA).$  Here we note that $L_a^2(dA_{\alpha})$ is a closed subspace of $L^2(\mathbb{D}, dA_{\alpha})$ and has the orthonormal basis $\{e_n\}_{n=0}^{\infty},$
where
$$e_n(z)=\sqrt{\frac{\Gamma(n+\alpha+2)}{n!\Gamma(\alpha+2)}}z^n~~\mbox{for all $n \geq 0$}.$$
For $f,g \in L_a^2(dA_{\alpha})$ the inner product of on $L_a^2(dA_{\alpha})$ can also be expressed as
\begin{align*}
	\langle f,g \rangle=\sum_{n=0}^{\infty}\frac{n!\Gamma(\alpha+2)}{\Gamma(n+\alpha+2)}\hat{f}_n\bar{\hat{g}}_n,
\end{align*}
where $f(z)=\sum_{n=0}^{\infty}\hat{f}_nz^n$ and $g(z)=\sum_{k=0}^{\infty}\hat{g}_nz^n.$
It is very well known that the weighted Bergman spaces are reproducing kernel Hilbert space and the reproducing kernel of $L_a^2(dA_{\alpha})$ at the point $w \in \mathbb{D}$ is given by
$$k^{\alpha}_w(z)=\frac{1}{(1-\bar{w}z)^{\alpha+2}}.$$ The normalized reproducing kernel at $w \in \mathbb{D}$ is given by
$$\hat{k}^{\alpha}_w(z)=\frac{(1-|w|^2)^{\frac{\alpha}{2}+1}}{(1-\bar{w}z)^{\alpha+2}}.$$
For more details about the weighted Bergman space we refer the book \cite{ZHU_BOOK}.

 Let $P_{\alpha}$ denote the orthogonal projection of $L^2(\mathbb{D}, dA_{\alpha})$ onto $L_a^2(dA_{\alpha}).$ Let $L^{\infty}(\mathbb{D}, dA_{\alpha})$ be the space of all complex measurable functions $\phi$ on $\mathbb{D}$ such that 
 $$\|\phi\|_{\infty,\alpha}=\sup\{ c\geq 0 : A_{\alpha}\left(\{z \in \mathbb{D} : |\phi(z)|>c\}\right)>0\}<\infty.$$
 For $\phi \in L^{\infty}(\mathbb{D}, dA_{\alpha}),$ the operator $T_{\phi}$ on $L_a^2(dA_{\alpha})$ defined by
 $$T_{\phi}f=P_{\alpha}(\phi f),\,\, f \in L_a^2(dA_{\alpha})$$
 is called the Toeplitz operator on $L_a^2(dA_{\alpha})$ with symbol $\phi.$ It is easy to observe that $T_{\phi}$ is a bounded linear operator on $L_a^2(dA_{\alpha})$ with $\|T_{\phi}\|\leq \|\phi\|_{\infty,\alpha}.$ Furthermore, if $\phi \in H^{\infty},$ then $T^*_{\phi}k^{\alpha}_w=\overline{\phi(w)}k^{\alpha}_w$ for all $w \in \mathbb{D},$ see \cite{FHS_CVEE_2015}.
 
Let $\phi : \mathbb{D} \to \mathbb{D}$ be an analytic self map on $\mathbb{D}$
 and $\psi \in H(\mathbb{D}).$ The weighted composition operator $C_{\psi,\phi } : H(\mathbb{D}) \to H(\mathbb{D})$ is defined by
 $$C_{\psi,\phi }f=\psi(f\circ \phi)~~\mbox{ for all $f \in H(\mathbb{D})$}.$$
 In particular, when $\phi$ is the identity mapping on $\mathbb{D}$ then $C_{\psi,\phi }$ becomes the multiplication operator $M_{\psi}$ and for $\psi=1,$ $C_{\psi,\phi }$ becomes the unweighted composition operator $C_{\phi}.$  In this article we limit our analysis of weighted composition operators on $L_a^2(dA_{\alpha}).$
 
In \cite[Th. 1]{CZ_IJM_2007}, Čučković et al. proved the following condition of boundedness of weighted composition operators on $L_a^2(dA_{\alpha}):$
 \begin{align}\label{E1B1}
    C_{\psi,\phi } \in \mathbb{B}(L_a^2(dA_{\alpha}))~~\mbox{if and only if}~~\sup_{a \in \mathbb{D}}I_{\phi, \alpha}(\psi)(a) <\infty,
 \end{align}
where 
$$I_{\phi, \alpha}(\psi)(a)=\int_{\mathbb{D}}\left(\frac{1-|a|^2}{|1-\bar{a}\phi(w)|^2}\right)^{\alpha+2}|\psi(w)|^2 dA_{\alpha}(w).$$
 Clearly, it follows from \eqref{E1B1} that if $\psi \in H^{\infty}$ then $C_{\psi,\phi } \in \mathbb{B}(L_a^2(dA_{\alpha})).$ Next relation follows from the reproducing property of $L_a^2(dA_{\alpha})$ that if $C_{\psi,\phi } \in \mathbb{B}(L_a^2(dA_{\alpha}))$ then 
$C^*_{\psi,\phi }k^{\alpha}_w=\overline{\psi(w)}k^{\alpha}_{\phi(w)}$ for all $w \in \mathbb{D}.$

\section{Shape of the numerical range}\label{S1}
We begin with the definition of $\alpha$-essential range.
\begin{definition}
For $\phi \in L^{\infty}(\mathbb{D}, dA_{\alpha})$ the $\alpha$-essential range of $\phi$ is denoted by  $R_{\phi, \alpha}$ and defined as 
$$R_{\phi, \alpha}=\{w : A_{\alpha}\left(\{z \in \mathbb{D} : |\phi(z)-w|<\epsilon\}\right)>0 ~~\mbox{for any $\epsilon >0$}\}.$$
\end{definition}
 It is easy to observe that $R_{\phi, \alpha}$ is a compact subset of $\mathbb{C}$ and 
$$\|\phi\|_{\infty, \alpha}=\max\{|w| : w \in R_{\phi, \alpha}\}.$$

Now, we prove the following lemma.

\begin{lemma}\label{L1}
	Let $\phi \in L^{\infty}(\mathbb{D}, dA_{\alpha})$ be such that $\phi$ is continuous on $\mathbb{D},$ then $\overline{\phi(\mathbb{D})}=R_{\phi, \alpha}.$
\end{lemma}

\begin{proof}
	If $u \in \phi(\mathbb{D})$ then there exists $z_0 \in \mathbb{D}$ such that $u=\phi(z_0).$ Since $\phi$ is continuous so for any $\epsilon>0$ there exists $\delta>0$ such that $|\phi(z)-u|<\epsilon$ whenever  $z \in S=\{z \in \mathbb{D} : |z-z_0|<\delta \}.$ Clearly, $A_{\alpha}(\{z \in \mathbb{D} : |\phi(z)-u|<\epsilon\}) \geq A_{\alpha}(S).$  If possible let $A_{\alpha}(S)=0.$ Then we have $\int_{S}(\alpha+1)(1-|z|^2)^{\alpha}dA=0.$ This implies that $(\alpha+1)(1-|z|^2)^{\alpha}=0$ almost all on $z \in S,$ which is not possible. Thus $A_{\alpha}(\{z \in \mathbb{D} : |\phi(z)-u|<\epsilon\}) \geq A_{\alpha}(S)>0.$ Therefore, $u \in R_{\phi, \alpha}$ and since $R_{\phi, \alpha}$ is closed so $\overline{\phi(\mathbb{D})}\subseteq R_{\phi, \alpha}.$

	Now, if $u \in R_{\phi, \alpha}$ then from the definition it follows that for any $\epsilon>0$ there exists $z \in \mathbb{D}$ such that $|\phi(z)-u|<\epsilon.$ Hence $u \in \overline{\phi(\mathbb{D})}$ and this completes the proof.
\end{proof}

\begin{remark}
	Here we note that if $\phi \in L^{\infty}(\mathbb{D}, dA_{\alpha})$ and $\phi$ is continuous on $\mathbb{D},$ then $\phi$ is bounded on $\mathbb{D}.$
\end{remark}




	





To prove our next result we need the following lemma which was proved in \cite{AFW_JFA_1993}.

\begin{lemma}\label{L2}
	If $\phi$ is harmonic and integrable over $\mathbb{D}$ then so is $\phi \circ \xi$ for every  Möbius transformation $\xi$ of $\mathbb{D},$ and $\int_{\mathbb{D}}(\phi \circ \xi) dA=\phi(\xi(0)).$
\end{lemma}

In the following proposition we obtain the spectrum of Toeplitz operator on $L_a^2(dA_{\alpha})$ with real and harmonic symbol, which generalizes the existing result on the Bergman space given in \cite[Prop. 12]{MS_IUMJ_1979}. 

\begin{proposition}\label{T1}
If $\phi \in L^{\infty}(\mathbb{D}, dA_{\alpha})$ is real and harmonic, then $\sigma(T_{\phi}) = [\inf\phi, \sup\phi].$
\end{proposition}

\begin{proof}
	First we have to prove  $\sigma(T_{\phi})\subseteq [\inf\phi, \sup\phi]$ and so we have to show that $T_{\phi-\lambda}$ is invertible whenever $\lambda \notin [\inf\phi, \sup\phi].$ Since $\lambda \notin [\inf\phi, \sup\phi]$ so either $\phi(z)-\lambda>0$ or $\phi(z)-\lambda<0$ for all $z \in \mathbb{D}.$ For the first case we choose $\epsilon>0$ such that 
	$$\|\epsilon(\phi-\lambda)-1\|_{\infty, \alpha}\leq \sup_{z \in \mathbb{D}}|\epsilon(\phi(z)-\lambda)-1|<1$$
	and so we have 
	$$\|T_{\epsilon(\phi-\lambda)}-I\|=\|T_{\epsilon(\phi-\lambda)-1}\|\leq\|\epsilon(\phi-\lambda)-1\|_{\infty, \alpha}<1.$$
	Thus $T_{\epsilon(\phi-\lambda)}$ is invertible and so $T_{\phi-\lambda}$ is invertible.
	For the second case $\phi(z)-\lambda<0$ for all $z \in \mathbb{D}$ implies that $-(\phi(z)-\lambda)>0$ for all $z \in \mathbb{D}.$ Proceeding similarly as the first case and using the relation $T_{\phi-\lambda}=-T_{-\phi+\lambda},$ we get the desired result.
	
	Next we prove the opposite inclusion. Since $\sigma(T_{\phi})$ is a closed subset of $\mathbb{C}$ so it suffices to show that $\sigma(T_{\phi})\supseteq (\inf\phi, \sup\phi).$ 
	As $\sigma(T_{\phi-\lambda})\supseteq (\inf\phi-\lambda, \sup\phi-\lambda)$ for any $\lambda \in \mathbb{R}$, so we only have to prove that $T_{\phi}$ is not invertible whenever $\phi$ takes both positive and negative values on $\mathbb{D}.$ Since $\phi$ is continuous so there exists $w \in \mathbb{D}$ such that $\phi(w)=0.$ Now, to reach our main goal we only show that $k^{\alpha}_w \notin {Range(T_{\phi})}.$ 
	We suppose that $k^{\alpha}_w \in Range(T_{\phi}),$ then there exists $f \in L_a^2(dA_{\alpha})$ such that $T_{\phi}f=k^{\alpha}_w$ i.e., $P_{\alpha}(\phi f)=k^{\alpha}_w.$ Since $\phi \in L^{\infty}(\mathbb{D},dA_{\alpha})$ so $\phi f \in L^2(\mathbb{D}, dA_{\alpha})=L_a^2(dA_{\alpha})\oplus L_a^2(dA_{\alpha})^{\perp}$ and we have $\psi \in L_a^2(dA_{\alpha})^{\perp}$ such that
	\begin{align}\label{L3E1}
		\phi f=k^{\alpha}_w+\psi.
	\end{align}
	For any $g \in H^{\infty},$ from \eqref{L3E1} we have 
	\begin{align}\label{L3E2}
		\int_{\mathbb{D}}\phi |f|^2\bar{g}dA_{\alpha}=\int_{\mathbb{D}}k^{\alpha}_w\bar{ f}\bar{g}dA_{\alpha}+\int_{\mathbb{D}}\psi\bar{f}\bar{g} dA_{\alpha}.
	\end{align}
	Since $\psi \in L_a^2(dA_{\alpha})^{\perp}$ and $fg \in L_a^2(dA_{\alpha})$ so 
	\begin{align*}
		\int_{\mathbb{D}}\psi\bar{f}\bar{g} dA_{\alpha}=0.
	\end{align*}
	
	Now, we get
	\begin{align*}
		\int_{\mathbb{D}}k^{\alpha}_w\bar{ f}\bar{g}dA_{\alpha}=\langle k^{\alpha}_w, fg \rangle=
		\langle k^{\alpha}_w, T_{g}f \rangle=\langle T_{g}^*k^{\alpha}_w, f \rangle=\overline{g(w)}\langle k^{\alpha}_w, f \rangle=\overline{g(w)f(w)}.
	\end{align*}
	Therefore, from \eqref{L3E2} we have
	\begin{align}\label{L3E3}
		\int_{\mathbb{D}}\phi |f|^2\bar{g}dA_{\alpha}=\overline{g(w)f(w)}.
	\end{align}
	In particular considering $g=1,$ the equality \eqref{L3E3} implies that $f(w) \in \mathbb{R}.$ Thus for any $g \in H^{\infty}$ from \eqref{L3E3} we get
	\begin{align}\label{L3E4}
		\int_{\mathbb{D}}\phi |f|^2 Re gdA_{\alpha}=f(w)Re(g(w)).
	\end{align}
	Since $Re H^{\infty}$ is weak$^*$-dense in the bounded real harmonic functions, then there exists a sequence $\{g_n\}$ in $Re H^{\infty}$ such that
	\begin{align}\label{T1E1}
		\lim_{n \to \infty}\int_{\mathbb{D}}\psi Re g_ndA=\int_{\mathbb{D}}\psi \phi dA~~\mbox{for all $\psi \in L^1(\mathbb{D},dA).$}
	\end{align}
	As $\phi$ is bounded and $f \in L_a^2(dA_{\alpha})$ so we have $(\alpha+1)\phi |f|^2(1-|z|^2)^{\alpha} \in L^1(\mathbb{D},dA).$ Thus from \eqref{T1E1}, we get
	\begin{align*}
		\lim_{n \to \infty}\int_{\mathbb{D}}\phi |f|^2 Re g_ndA_{\alpha}=\int_{\mathbb{D}}\phi^2 |f|^2dA_{\alpha}.
	\end{align*}
	By applying \eqref{L3E4} we obtain that
	\begin{align}\label{T1E2}
		\lim_{n \to \infty}f(w)Re g_n(w)=\int_{\mathbb{D}}\phi^2 |f|^2dA_{\alpha}.
	\end{align}
	Let $\hat{k}_w$ be the normalized reproducing kernel of $L_a^2(dA)$ at the point $w.$ Then the function $|\hat{k}_w(z)|^2 \in L^1(\mathbb{D},dA)$ and from \eqref{T1E1}, we get
	\begin{align}\label{T1E3}
		\lim_{n \to \infty}\int_{\mathbb{D}}Re g_n |\hat{k}_w(z)|^2 dA=\int_{\mathbb{D}}\phi|\hat{k}_w(z)|^2dA.
	\end{align}
	Let $\phi_w$ be the M\"obius map on $\mathbb{D},$ given by $\phi_w(z)=\frac{w-z}{1-\bar{w}z}$ for all $z \in \mathbb{D}.$ Since the real Jacobian of $\phi_w$ is given by $|\hat{k}_w(z)|^2,$ we have
	\begin{align}\label{T1E4}
	  \lim_{n \to \infty}\int_{\mathbb{D}}Re g_n \circ \phi_w dA=\int_{\mathbb{D}}\phi\circ \phi_wdA.
	\end{align}
	Now, by Lemma \ref{L2} and \eqref{T1E4}, we obtain
	\begin{align}\label{T1E5}
		\lim_{n \to \infty}Re g_n(w)=\phi(w).
	\end{align}
	Therefore, combining \eqref{T1E2} and \eqref{T1E5} we get
	\begin{align*}
		\int_{\mathbb{D}}\phi^2 |f|^2dA_{\alpha}=f(w)\phi(w)=0.
	\end{align*}
	This implies that $\phi^2 |f|^2\equiv 0$ on $\mathbb{D}.$ As $\phi$ takes positive value on $\mathbb{D}$ and $f$ is analytic on $\mathbb{D}$ so we have $f=0.$ This implies that $T_{\phi}f=k^{\alpha}_w=0,$ which is a contradiction as $k^{\alpha}_w \neq 0.$ Thus, $k^{\alpha}_w \notin {Range(T_{\phi})}$ and this completes the proof.
	
\end{proof}

In the next result we completely determine the numerical range of Toeplitz operators acting on $L^2_a(dA_{\alpha})$ with harmonic symbol.

\begin{theorem}\label{T3}
	Let $\phi \in L^{\infty}(\mathbb{D}, dA_{\alpha})$ be a nonconstant and harmonic function on $\mathbb{D},$ then $W(T_{\phi})=$ $Rel$ $int$ $\overline{\phi(\mathbb{D})}^{\wedge}.$
\end{theorem}

\begin{proof}
	We prove this theorem with the following two cases.\\
	\textbf{Case 1} First we prove this result when $\phi \in L^{\infty}(\mathbb{D}, dA_{\alpha})$ is a nonconstant, real-valued and harmonic function on $\mathbb{D}.$ Since $\phi$ is real-valued function so $T_{\phi}$ is self adjoint and by Proposition \ref{T1} we get that $\overline{W(T_{\phi})}=\sigma(T_{\phi})^{\wedge}=[\inf \phi, \sup \phi].$ As $W(T_{\phi})$ is convex so $(\inf \phi, \sup \phi) \subseteq W(T_{\phi}).$ Now, we will show that $T_{\phi}$ has no eigenvector. As $T_{\phi-\lambda}=T_{\phi}-\lambda I$ for all $\lambda \in \mathbb{C}$ so we have to show for any $f \in L_a^2(dA_{\alpha}),$ $T_{\phi}f=0$ implies that $f=0.$ If $T_{\phi}f=0$ then $\phi f \in L_a^2(dA_{\alpha})^{\perp}.$ For any $g \in H^{\infty}$ we have $fg \in  L_a^2(dA_{\alpha})$. Thus we get
	\begin{align*}
		\int_{\mathbb{D}}\phi|f|^2\bar{g}dA_{\alpha}=\langle \phi f,fg \rangle=0.
	\end{align*}
	Hence, we obtain
	\begin{align*}
		\int_{\mathbb{D}}\phi|f|^2 Re gdA_{\alpha}=\int_{\mathbb{D}}\phi|f|^2Re\bar{g}dA_{\alpha}=0.
	\end{align*}
	Now, proceeding similarly as Proposition \ref{T1} we get
	\begin{align*}
		\int_{\mathbb{D}}\phi^2|f|^2dA_{\alpha}=0.
	\end{align*}
	This implies that $f=0.$
	Thus $\inf \phi, \sup \phi \notin W(T_{\phi})$ because if either $\inf \phi$ or  $\sup \phi$ are in $W(T_{\phi})$ then they are extreme points of $T_{\phi}$ and hence they are eigenvalues of $T_{\phi}.$ This completes the proof for the first case.\\
\textbf{Case 2} Now, we prove this result when $\phi \in L^{\infty}(\mathbb{D}, dA_{\alpha})$ is a nonconstant, complex-valued and harmonic function on $\mathbb{D}.$ We first prove the inclusion $W(T_{\phi})\subseteq $ $Rel$ $int$ $\overline{\phi(\mathbb{D})}^{\wedge}.$
Let $M_{\phi}$ be the multiplication operator on $L^2(\mathbb{D},dA_{\alpha}).$ Since $M_{\phi}$ is a normal operator and $\sigma(M_{\phi})=R_{\phi,\alpha}$ (see \cite[Prob. 67]{H_BOOK}), so 
$\overline{W(M_{\phi})}=R^{\wedge}_{\phi,\alpha}.$ 
Therefore, from Lemma \ref{L1} we get 
\begin{align}\label{T3E1}
	\overline{W(M_{\phi})}=\overline{\phi(\mathbb{D})}^{\wedge}.
\end{align}
As $T_{\phi}$ dilates to $M_{\phi}$ so
\begin{align}\label{T3E2}
	W(T_{\phi}) \subseteq \overline{W(M_{\phi})}.
\end{align}
Now, combining \eqref{T3E1} and \eqref{T3E2} we obtain 
\begin{align}\label{T3E3}
	W(T_{\phi}) \subseteq \overline{\phi(\mathbb{D})}^{\wedge}.
\end{align}
Suppose that $W(T_{\phi})$  is not contained in  $Rel$ $int$ $\overline{\phi(\mathbb{D})}^{\wedge}.$ Then there exists $\theta \in \mathbb{R},$ $\gamma \in \mathbb{C}$ and $f \in L_a^2(dA_{\alpha})$ with $\|f\|=1$ such that 
$$\langle T_{Re(e^{i \theta}(\phi +\gamma))}f,f \rangle=\max \overline{Re(e^{i \theta}(\phi +\gamma))(\mathbb{D})}^{\wedge}=\delta.$$
This implies that 
$$\langle M_{Re(e^{i \theta}(\phi +\gamma))}f,f \rangle=\langle {Re(e^{i \theta}(\phi +\gamma))}f,f \rangle=\delta.$$
Since $M_{Re(e^{i \theta}(\phi +\gamma))} \leq \delta I$ so we get $Re(e^{i \theta}(\phi +\gamma))f=cf.$ The analyticity of non zero $f$ implies that $Re(e^{i \theta}(\phi +\gamma)(z))=\delta$ for all $z \in \mathbb{D}.$ Therefore, $\overline{(\phi +\gamma)(\mathbb{D})}^{\wedge}$ is contained in a line. Repeating the above process with $Im(e^{i \theta}(\phi +\gamma))$ yields that $\phi$ is constant, which contradicts our assumption. Therefore, $W(T_{\phi})$ is contained in the relative interior of $\overline{\phi(\mathbb{D})}^{\wedge}.$

If possible let they are not equal then there exists $\theta \in \mathbb{R}$ and $c \in \mathbb{C}$ such that $W(T_{\psi})\subsetneqq$ $Rel$ $int$ $\overline{\psi(\mathbb{D})}^{\wedge},$ where $\psi=Re(e^{i \theta}(\phi+c))$ which is real and harmonic and this contradicts {Case 1}. Thus we obtain the desired relation.
\end{proof}

The following corollary follows from Theorem \ref{T3}.
\begin{cor}\label{C1}
	If $\phi \in H^{\infty}$ then $W(M_{\phi})=\phi(\mathbb{D})^{\wedge}.$
\end{cor}

The next example demonstrates that the condition $\phi$ is harmonic in $\mathbb{D}$ is necessary in Theorem \ref{T3}.

\begin{example}
	If we consider the function $\phi(z)=|z|^2$ on $\mathbb{D}$ then $\phi$ is continuous but not harmonic in $\mathbb{D}.$ For any $n,m\geq 0,$ we have
	\begin{align*}
	&\langle T_{\phi}e_n,e_m
	\rangle \\
	&=\frac{\sqrt{\Gamma(n+\alpha+2)\Gamma(m+\alpha+2)}}{\sqrt{n!m!}\Gamma(\alpha+2)}(\alpha+1)\int_{\mathbb{D}}\phi(z)z^n\bar{z}^mdA_{\alpha}(z)\\
	&=\frac{\sqrt{\Gamma(n+\alpha+2)\Gamma(m+\alpha+2)}}{\sqrt{n!m!}\Gamma(\alpha+1)}\left(\int_{r=0}^1 r^{n+m+3}(1-r^2)^{\alpha}dr \right)\left(\int_{\theta=0}^{2\pi} e^{i(n-m)\theta}d\theta \right)\\
	&=\begin{cases}
		\lambda_n\,\,\,\,\mbox{if $n=m$}\\
		0\,\,\,\,\mbox{ if $n \neq m$}
	\end{cases},
	\end{align*}
where 
\begin{align*}
	\lambda_n =\frac{\pi\Gamma(n+\alpha+2)}{n!\Gamma(\alpha+1)}\int_{r=0}^1 r^{n+1}(1-r)^{\alpha}dr= \frac{\pi(n+1)}{n+\alpha+2}.
\end{align*}
Clearly, $\{\lambda_n\}_{n=0}^\infty$ is a increasing sequence with $\lambda_n \to \pi.$ Thus the matrix representation of $T_{\phi}$ is a diagonal matrix with diagonal elements $\lambda_n,$ relative to the standard ordered basis $\{e_n\}_{n=0}^\infty$ of $L_a^2(dA_{\alpha}).$ Hence $W(T_{\phi})=\left[\frac{\pi}{\alpha+2},\pi\right).$ Again  $Rel$ $int$ $\overline{\phi(\mathbb{D})}^{\wedge}=(0,1).$ Hence for this example $W(T_{\phi})$ is not equal to $Rel$ $int$ $\overline{\phi(\mathbb{D})}^{\wedge}$ furthermore $W(T_{\phi})$ is not a relatively open subset of $\mathbb{C}.$ 
\end{example}

We now introduce the following definition.

\begin{definition}
	For $n \geq 2$ and $n>j\geq 0,$ we define the set $L_j$ of $L_a^2(dA_{\alpha})$ as
	$$\\L_j=\left\{ f \in L_a^2(dA_{\alpha}) : f(z)=z^jg(z^n), g \in L_a^2(dA_{\alpha}) \right\}.$$
\end{definition}

Now, we prove the following lemma which will be useful to prove the next result.
\begin{lemma}\label{L11}
	If $m \in \mathbb{N}$ and $c>1$ then the sequence $\left\{x_n=\frac{n!\Gamma(nm+c)}{(nm)!\Gamma(n+c)}\right\}$ is bounded.
\end{lemma}

\begin{proof}
	If $c$ is an integer then it easily follows that $\{x_n\}$ is bounded. Now, the function $f(x)=\frac{\Gamma(mn+x)}{\Gamma(n+x)}$ is increasing on the interval $[1,\infty).$ So, for an arbitrary $c>1$ by choosing an integer greater than $c$ and using the boundedness of the sequence for the integer case the desired result follows.
\end{proof}

Next we prove the following decomposition of $L_a^2(dA_{\alpha}),$ which will be an essential tool to compute the numerical range of weighted composition operators on $L_a^2(dA_{\alpha}).$

\begin{proposition}
	For each $n \geq 2,$ $L_a^2(dA_{\alpha})$ can be decomposed as
	$$L_a^2(dA_{\alpha})=L_0 \oplus L_1 \oplus \ldots \oplus L_{n-1}.$$
\end{proposition}

\begin{proof}
	By applying Lemma \ref{L11} it follows that each element of $L_a^2(dA_{\alpha})$ of the form $\sum_{k=0}^{\infty}a_kz^{kn+j}$ lies in $L_j$ and conversely. Then it follows easily that each $L_j$ is a closed subspace and for all $n \geq 2,$
	$$L_a^2(dA_{\alpha})=L_0 \oplus L_1 \oplus \ldots \oplus L_{n-1}.$$ 
\end{proof}

Our next result read as:

\begin{lemma}\label{Lem1}
	If $\psi \in H^{\infty}$ and $M_{\psi} (L_j) \subseteq L_j$ then $W(M_{\psi}|_{L_j})=\phi(\mathbb{D})^{\wedge}.$
\end{lemma}

\begin{proof}
Let $p_j$ be the orthogonal projection from $L_a^2(dA_{\alpha})$ onto $L_j.$ Now, for $w \in \mathbb{D}\setminus\{0\}$ we denote $k^{\alpha}_{w,j}=p_jk^{\alpha}_w.$ Then we have 
	\begin{align*}
		\langle M_{\psi}\hat{k}^{\alpha}_{w,j},\hat{k}^{\alpha}_{w,j} \rangle =\frac{1}{\|k^{\alpha}_{w,j}\|^2}\langle \psi k^{\alpha}_{w,j},k^{\alpha}_{w,j}  \rangle=\frac{1}{\|k^{\alpha}_{w,j}\|^2} \psi(w) k^{\alpha}_{w,j}(w)=\psi(w).
	\end{align*}
Thus for any $w \in \mathbb{D}\setminus\{0\},$ $\psi(w) \in W(M_{\psi}|_{L_j})$ and $\psi(0) \in W(M_{\psi}|_{L_j})$ by open mapping theorem. Thus we get 
$W(M_{\psi}|_{L_j})  \supseteq \phi(\mathbb{D})^{\wedge}.$ Now, the desired result follows from the Corollary \ref{C1} and $W(M_{\psi}|_{L_j}) \subseteq W(M_{\psi}).$
\end{proof}

Now, we are in a position to prove the following result.

\begin{theorem}\label{Th1}
	Let $\phi(z)=\lambda z$ with $\lambda=e^{2\pi i/n}$ and $\psi(z)=g(z^n)$ for some $g \in H^{\infty}.$ Then 
	$$W(C_{\psi,\phi})=(\psi(\mathbb{D})\cup \lambda \psi(\mathbb{D}) \cup \ldots \cup \lambda^{n-1}\psi(\mathbb{D}))^{\wedge}.$$
\end{theorem}

\begin{proof}
	If $f \in L_j$ then $f(\phi(z))=f(\lambda z)=\lambda^jf(z)$ and we have $C_{\phi}(L_j)\subseteq L_j.$ Since $\psi(z)=g(z^n)$ and  $\psi$ is bounded on $\mathbb{D}$ so $M_{\psi}(L_j)\subseteq L_j$ and so $C_{\psi,\phi}(L_j) \subseteq L_j.$ This implies that 
	$$C_{\psi,\phi}(L_j)=C_0\oplus C_1 \oplus \ldots \oplus C_{n-1},$$
	where $C_j=C_{\psi,\phi}|_{L_j}.$ For any $h \in L_j$ with $\|h\|=1$ we have
	$\langle C_jh,h \rangle=\lambda^j \langle \psi h,h \rangle$ and this implies that $W(C_j)=\lambda^jW(M_{\psi}|_{L_j}).$ Hence from the Lemma \ref{Lem1} we have $W(C_j)=\lambda^j\phi(\mathbb{D})^{\wedge}.$ Thus 
	\begin{align*}
		W(C_{\psi,\phi})&=(W(C_0)\cup W(C_1)\cup \ldots \cup W(C_{n-1}))^{\wedge}\\
		&=(\psi(\mathbb{D})\cup \lambda \psi(\mathbb{D}) \cup \ldots \cup \lambda^{n-1}\psi(\mathbb{D}))^{\wedge},
	\end{align*}
as desired.
\end{proof}

The following corollary follows immediately from Theorem \ref{Th1}.

\begin{cor}\label{C2}
	If $\phi(z)=-z$ then $W(C_{\phi})=[-1,1]$ and if $\phi(z)=e^{2\pi i/n}z$ with $n>2$ then $W(C_{\phi})$ is the closed, regular polygonal region with $n$ sides and inscribed in the unit circle. 
\end{cor}

In the following, we classify some subsets of $\mathbb{C}$ which are the numerical range of weighted composition operators acting on $L_a^2(dA_{\alpha}).$ To do this we start with the following definition.
 
A subset $S$ of $\mathbb{C}$ is said to have $n$-fold symmetry about the origin if it satisfies $e^{2\pi i/n}S=S.$

\begin{theorem}\label{Th2}
	Let $S$ be an non-empty, open, bounded and convex subset of $\mathbb{C}.$ If $S$ has $n$-fold symmetry about the origin then  for any $\alpha>-1$ there exists $C_{\psi,\phi} \in  \mathbb{B}\left({L_a^2(dA_{\alpha})}\right)$ such that $W(C_{\psi,\phi})=S.$
\end{theorem}

\begin{proof}
	Let $f$ be a Riemann map from $\mathbb{D}$ onto $S,$ and we consider the function $\psi : \mathbb{D} \to S$ such that $\psi(z)=f(z^n)$ for all $z \in \mathbb{D}.$ It is easy to verify that $\psi(\mathbb{D})=f(\mathbb{D})$ and so we obtain $\psi(\mathbb{D})=S.$ Let $\phi(z)=\lambda z$ where $\lambda=e^{2\pi i/n}.$ By applying Theorem \ref{Th1} we get 
	\begin{align}\label{Th2e1}
	W(C_{\psi,\phi})=(\psi(\mathbb{D})\cup \lambda \psi(\mathbb{D}) \cup \ldots \cup \lambda^{n-1}\psi(\mathbb{D}))^{\wedge}.
	\end{align}
    Since $S$ is a $n$-fold symmetry about the origin so $\lambda^k\psi(\mathbb{D})=S$ for all $0 \leq k \leq n-1.$ As $S$ is convex so from \eqref{Th2e1} we obtain the the desired result.
\end{proof}

\begin{cor}
	Let $f \in H^{\infty}$ be nonconstant. If $n>1$ then for any $\alpha>-1$ there exists $C_{\psi,\phi} \in  \mathbb{B}\left({L_a^2(dA_{\alpha})}\right)$ such that  $W(C_{\psi,\phi})$ is the smallest convex set with $n$-fold symmetry about the origin with 
	$W(C_{\psi,\phi}) \supseteq f(\mathbb{D}).$
\end{cor}

\begin{proof}
	Let us define $\psi : \mathbb{D} \to \mathbb{C}$ such that $\psi(z)=f(z^n)$ for all $z \in \mathbb{D}.$ Then $\psi(\mathbb{D})=f(\mathbb{D}).$ Let $\phi(z)=\lambda z$ where $\lambda=e^{2\pi i/n}.$ From Theorem \ref{Th1} we get 
	\begin{align}\label{Th2e2}
		W(C_{\psi,\phi})=(f(\mathbb{D})\cup \lambda f(\mathbb{D}) \cup \ldots \cup \lambda^{n-1}f(\mathbb{D}))^{\wedge}.
	\end{align}
 As $\lambda W(C_{\psi,\phi})=W(C_{\psi,\phi})$ so $W(C_{\psi,\phi})$ has a $n$-fold symmetry about the origin and it is a convex set. If $M$ is a convex set and has a $n$-fold symmetry about the origin with $M \supseteq f(\mathbb{D})$ then $M \supseteq \lambda^k f(\mathbb{D})$ for all $0\leq k \leq n-1.$ Thus from \eqref{Th2e2} it follows that $M \supseteq W(C_{\psi,\phi}).$ Therefore, $W(C_{\psi,\phi})$ is the smallest convex set with $n$-fold symmetry about the origin with $W(C_{\psi,\phi}) \supseteq f(\mathbb{D}).$
\end{proof}

We conclude this section by asking which non-empty, bounded, and convex subsets of $\mathbb{C}$ can be the numerical range of weighted composition operators acting on $L_a^2(dA_{\alpha}).$

\section{Containment of zero in the numerical range}\label{S2}

Our main focus in this section is to investigate the containment of the origin in the interior of the numerical range of weighted composition operators on $L_a^2(dA_{\alpha}).$ In the first result we study the case when the origin is contained in the numerical range as well as in its closure for the sum of two weighted composition operators. To do so we recall the definition of radical limit.

A function $f \in H(\mathbb{D})$ is said to be has a radical limit if $\lim_{r \to 1} f(re^{i \theta})$ exists almost everywhere in $\partial \mathbb{D}.$
It is proved in \cite[Th. 11.32]{RUDIN_BOOK} that for every $f \in H^{\infty}$ there corresponds a function $f^* \in L^{\infty}(\partial \mathbb{D}),$ defined almost everywhere by
\begin{align*}
	f^*(e^{i \theta})=\lim_{r \to 1} f(re^{i \theta}).
\end{align*}
Moreover, if $f^*(e^{i \theta})=0$ for almost all $e^{i \theta}$ on some arc $I \subseteq \partial \mathbb{D},$ then $f(z)=0$ for all $z \in \mathbb{D}.$

Now, we are in a position to prove the first result of this section.

\begin{theorem}\label{Theo1}
	Let $\phi_1$ and $\phi_2$ be two holomorphic self maps on $\mathbb{D}$ and $\psi_1,\psi_2 \in H(\mathbb{D})$ be such that $C_{\psi_1,\phi_1},C_{\psi_2,\phi_2} \in  \mathbb{B}\left({L_a^2(dA_{\alpha})}\right).$ 
	\begin{align*}
		&(i)~~ \mbox{If $\phi_1$ and $\phi_2$ are identity maps on $\mathbb{D},$ and $\phi_1, \phi_2$ have a common zero in $\mathbb{D},$ }\\
		& ~~\mbox{then $0 \in W(C_{\psi_1,\phi_1}+C_{\psi_2,\phi_2}).$}\\ 
		&(ii)~~ \mbox{If $\phi_1,\phi_2$ are not identity maps on $\mathbb{D}$ and $\psi_1, \psi_2 \in H^{\infty},$ then $0 \in \overline{W(C_{\psi_1,\phi_1}+C_{\psi_2,\phi_2})}.$} 
	\end{align*}
\end{theorem}

\begin{proof}
	For any $w \in \mathbb{D},$ we have
	\begin{align}\label{Theo1e1}
		&\langle (C_{\psi_1,\phi_1}+C_{\psi_2,\phi_2})\hat{k}^{\alpha}_w, \hat{k}^{\alpha}_w\rangle \nonumber\\
		&=\frac{1}{\|{k}^{\alpha}_w\|^2}\langle {k}^{\alpha}_w, (C_{\psi_1,\phi_1}+C_{\psi_2,\phi_2})^*{k}^{\alpha}_w\rangle \nonumber\\
		&=\frac{1}{\|{k}^{\alpha}_w\|^2}\left(\langle {k}^{\alpha}_w, \overline{\psi_1(w)} {k}^{\alpha}_{\phi_1(w)}\rangle+\langle {k}^{\alpha}_w, \overline{\psi_2(w)} {k}^{\alpha}_{\phi_2(w)}\rangle\right) \nonumber\\
		&=\frac{\psi_1(w)}{\|{k}^{\alpha}_w\|^2}{k}^{\alpha}_{w}(\phi_1(w))+\frac{\psi_2(w)}{\|{k}^{\alpha}_w\|^2}{k}^{\alpha}_{w}(\phi_2(w)) \nonumber\\
		&=\frac{\psi_1(w)(1-|w|^2)^{\alpha+2}}{(1-\bar{w}\phi_1(w))^{\alpha+2}}+\frac{\psi_2(w)(1-|w|^2)^{\alpha+2}}{(1-\bar{w}\phi_2(w))^{\alpha+2}}.
	\end{align}
	$(i)$ Since $\phi_1$ and $\phi_2$ are identity maps on $\mathbb{D}$ then from \eqref{Theo1e1} we have 
	$$\langle (C_{\psi_1,\phi_1}+C_{\psi_2,\phi_2})\hat{k}^{\alpha}_w, \hat{k}^{\alpha}_w\rangle=\psi_1(w)+\psi_2(w).$$ If $\psi_1(w_0)=\psi_2(w_0)=0$ for $w_0 \in \mathbb{D}$ then we get $$\langle (C_{\psi_1,\phi_1}+C_{\psi_2,\phi_2})\hat{k}^{\alpha}_{w_0}, \hat{k}^{\alpha}_{w_0}\rangle=0,$$ as desired.\\
	$(ii)$ As $\phi_1,\phi_2$ are not identity maps then the sets $\{e^{i \theta} : \phi^*_1(e^{i \theta})=e^{i \theta}\}$ and $\{e^{i \theta} : \phi^*_2(e^{i \theta})=e^{i \theta}\}$ have measure zero on $\partial \mathbb{D}.$  So there exists a $w_0 \in \partial \mathbb{D}$ such that $\psi^*_i(w_0)$ exists and $\phi^*_i(w_0) \neq w_0$ for $i=1,2.$ Hence from \eqref{Theo1e1} we have 
	$$\lim_{w \to w_0}\langle (C_{\psi_1,\phi_1}+C_{\psi_2,\phi_2})\hat{k}^{\alpha}_w, \hat{k}^{\alpha}_w\rangle =0.$$ Therefore,
	$ 0 \in \overline{W(C_{\psi_1,\phi_1}+C_{\psi_2,\phi_2})}$ and this completes the proof.
\end{proof}

Next, we completely characterize the numerical range of the bounded weighted composition operators induced by constant composition maps.

\begin{proposition}\label{pro1}
	Let $\phi \equiv w$ with $|w|<1$ be such that $C_{\psi,\phi} \in  \mathbb{B}\left({L_a^2(dA_{\alpha})}\right).$ 
	\begin{align*}
		&(i)~~ \mbox{If $k^{\alpha}_w=\mu \psi$ for some $\mu \neq 0$ then $W(C_{\psi,\phi})=[0, \bar{\mu}\|\psi\|^2].$}\\ 
		&(ii)~~ \mbox{If $k^{\alpha}_w \perp \psi$ then $W(C_{\psi,\phi})$ is the closed disc centred at the origin and radius $\frac{\|\psi\|}{2(1-|w|^2)^{\frac{\alpha}{2}+1}}.$ }\\
		&(iii)~~ \mbox{Otherwise $W(C_{\psi,\phi})$ is a closed ellipse with foci at $0$ and $\psi(w).$} 
	\end{align*}
\end{proposition}

\begin{proof}
	Since $\phi$ is a constant function and $C_{\psi,\phi} \in  \mathbb{B}\left({L_a^2(dA_{\alpha})}\right)$ so for any $f \in L_a^2(dA_{\alpha})$
	\begin{align*}
		C_{\psi,\phi}f=\psi f(w)=\langle f, k^{\alpha}_w \rangle \psi.
	\end{align*}
	Thus $C_{\psi,\phi}$ is a rank one operator and the desired result follows from \cite[Prop. 2.5]{BS_IEOT_2002}.
\end{proof}

Now, we state the following lemma which follows from \cite[Th. 2.6]{KEY_JME_2021}.
\begin{lemma}\label{lem1}
	Let $C_{\psi,\phi} \in  \mathbb{B}\left({L_a^2(dA_{\alpha})}\right),$ $\phi$ be a nonconstant analytic self map on $\mathbb{D}$ and $\psi$ be non-zero. If either 
	$\psi$ has a zero on $\mathbb{D}$ or $\phi$ is not one-to-one, then $0 \in$ $int$ $W(C_{\psi,\phi}).$
\end{lemma}

	

In the following results we investigate the inclusion of the origin in the interior of numerical range of weighted composition operators on $L_a^2(dA_{\alpha}).$

\begin{theorem}\label{Theo2}
	Let $C_{\psi,\phi} \in  \mathbb{B}\left({L_a^2(dA_{\alpha})}\right)$ and $\phi(0)=0.$ If $\phi$ is not of the form $\phi(z)=tz$ where $t \in \overline{\mathbb{D}},$ then $0 \in$ $int$ $W(C_{\psi,\phi}).$
\end{theorem}

\begin{proof}
	If $\phi^{\prime}(0)=0$ then $\phi$ is not one-to-one and by Lemma \ref{lem1} we have  $0 \in$ $int$ $W(C_{\psi,\phi}).$ We now consider that $\phi^{\prime}(0)=\lambda \neq 0.$ As $\phi$ is not of the form $\phi(z)=tz$ where $t \in \overline{\mathbb{D}},$ so $\phi$ can be written as
	\begin{align*}
		\phi(z)=\lambda z \left( 1+bz^m(1+h(z)) \right),
	\end{align*} 
	where $m$ is a positive integer, $b \neq 0$ and $h \in H(\mathbb{D})$ with $h(0)=0.$ Therefore, for any $n \geq 1$
	\begin{align}\label{Theo2e1}
		\phi^n(z)=\lambda^n z^n+nb\lambda^nz^{n+m}+ \mbox{higher order terms of $z$}.
	\end{align}
	The matrix of $C_{\psi,\phi}$ with respect to the orthonormal basis $\{e_n\}^{\infty}_{n=0}$ has its $n$-th column given by the sequence of coefficients of the power series expansion of $\sqrt{\frac{\Gamma(n+\alpha+2)}{n!\Gamma(\alpha+2)}}\psi\phi^n.$
	Let $M_n$ be the subspace of $L_a^2(dA_{\alpha})$ spanned by $e_n$ and $e_{n+m}.$ Clearly, $M_n$ is a two dimensional subspace of $L_a^2(dA_{\alpha}).$ Let $\sum_{k=0}^{\infty}\hat{\psi}_kz^k$ be the power series expansion of $\psi.$ Then the matrix representation of $C_{\psi,\phi}$ on $M_n$ with respect to the basis $\{e_n,e_{n+m}\}$ is given by
	\begin{align*} 
		\begin{pmatrix}
			\hat{\psi}_0\lambda^n&0\\
			\sqrt{\frac{(m+n)!\Gamma(n+\alpha+2)}{n!\Gamma(m+n+\alpha+2)}}\lambda^n(nb\hat{\psi}_0+\hat{\psi}_m) &\hat{\psi}_0\lambda^{m+n}
		\end{pmatrix}=\lambda^nC_n,
	\end{align*}
	where 
	\begin{align*} 
		C_n=\begin{pmatrix}
			\hat{\psi}_0&0\\
			\sqrt{\frac{(m+n)!\Gamma(n+\alpha+2)}{n!\Gamma(m+n+\alpha+2)}}(nb\hat{\psi}_0+\hat{\psi}_m) &\hat{\psi}_0\lambda^{m}
		\end{pmatrix}.
	\end{align*}
	Since the numerical range of compression is contained in the numerical range of the operator (see \cite[Prop. 1.4]{wu_gau_book}) so it is sufficient to show that $0 \in$ $int$ $W(C_n)$ for some $n.$ If $\hat{\psi}_0=0$ then it follows from Lemma \ref{lem1} that $0 \in$ $int$ $W(C_n).$ If $\hat{\psi}_0 \neq 0$ then $W(C_n)$ is an ellipse with foci $\hat{\psi}_0$ and $\hat{\psi}_0\lambda^{m}$ with the length of the minor axis $\sqrt{\frac{(m+n)!\Gamma(n+\alpha+2)}{n!\Gamma(m+n+\alpha+2)}}|nb\hat{\psi}_0+\hat{\psi}_m|,$ see \cite[Example 3]{GR_BOOK}. A simple computation shows that 
	$$\lim_{n \to \infty}\frac{(m+n)!\Gamma(n+\alpha+2)}{n!\Gamma(m+n+\alpha+2)}=1.$$
	Thus if we choose $n$ large enough then the length of the minor axis of $W(C_n)$ will larger than the modulus of its centre. Hence there exits $n$ for which $0 \in$ $int$ $W(C_n),$ as desired.
\end{proof}

\begin{theorem}\label{Theo3}
	Let $C_{\psi,\phi} \in  \mathbb{B}\left({L_a^2(dA_{\alpha})}\right)$ and $\psi$ be nonconstant. If $\phi(z)=tz$ where $-1\leq t\leq 0,$ then $0 \in$ $int$ $W(C_{\psi,\phi}).$
\end{theorem}

\begin{proof}
	First we prove this result for $\psi(0)=0.$ For $-1\leq t< 0$ the result follows from Lemma \ref{lem1} and for $t=0$ the Proposition \ref{pro1} implies that $0 \in$ $int$ $W(C_{\psi,\phi}).$
	
	Now to prove the result for $\psi(0)\neq 0$ it is enough to show this for $\psi(0)=1.$
	So, $\psi(z)=1+\eta(z),$ where $\eta$ is a nonconstant analytic function with $\eta(0)=0.$
	Now, we show that $C_{\eta,\phi} \in  \mathbb{B}\left({L_a^2(dA_{\alpha})}\right).$ Since $\eta(z)=\psi(z)-1$ so we get
	$$|\eta(z)|^2 \leq (|\psi(z)|+1)^2 \leq 2(|\psi(z)|^2+1).$$
	This implies that 
	\begin{align}\label{Theo3E1}
	I_{\phi,\alpha}(\eta)(a) \leq 2 \left(I_{\phi,\alpha}(\psi)(a)+\int_{\mathbb{D}}\left(\frac{1-|a|^2}{|1-\bar{a}\phi(w)|^2}\right)^{\alpha+2}dA_{\alpha}(w)\right),
	\end{align}
	where $$I_{\phi,\alpha}(\eta)(a)=\int_{\mathbb{D}}\left(\frac{1-|a|^2}{|1-\bar{a}\phi(w)|^2}\right)^{\alpha+2}|\eta(w)|^2 dA_{\alpha}(w)$$ 
	and $$I_{\phi,\alpha}(\psi)(a)=\int_{\mathbb{D}}\left(\frac{1-|a|^2}{|1-\bar{a}\phi(w)|^2}\right)^{\alpha+2}|\psi(w)|^2 dA_{\alpha}(w).$$
	As $C_{\psi,\phi} \in  \mathbb{B}\left({L_a^2(dA_{\alpha})}\right)$ so by \eqref{E1B1} and from \eqref{Theo3E1} we obtain $\sup_{a \in \mathbb{D}}I_{\eta, \alpha}(\psi)(a) <\infty.$ Thus we have $C_{\eta,\phi} \in  \mathbb{B}\left({L_a^2(dA_{\alpha})}\right).$
	Now, for any $f \in L_a^2(dA_{\alpha})$ with $\|f\|=1,$ we get
	\begin{align*}
		\langle C_{\psi,\phi}f,f \rangle=\langle C_{\phi}f,f \rangle+\langle C_{\eta,\phi}f,f \rangle.
	\end{align*}
	For $t=0$ then the result follows from Proposition \ref{pro1}. Since $\eta(0)=0,$ for $-1 \leq t <0,$ it follows from Lemma \ref{lem1} that $W(C_{\eta,\phi})$ contains a disc of positive radius and centred at the origin. So there exists $f_1 \in L_a^2(dA_{\alpha})$ with $\|f_1\|=1$ such that $Im \langle C_{\eta,\phi}f_1,f_1 \rangle>0.$ As $\langle C_{\phi}f_1,f_1 \rangle$ is real so $p_1=\langle C_{\psi,\phi}f_1,f_1 \rangle$ is in the upper half plane. Similarly, we get another point $p_2$ in the lower half plane. Again we have $\langle C_{\psi,\phi}e_1,e_1 \rangle=t$ and $\langle C_{\psi,\phi}e_0,e_0 \rangle=1.$ Thus $0 \in$ $int$ $\{p_1,p_2,t,1\}^{\wedge} \subseteq$ $int$ $W(C_{\psi,\phi}),$ as desired.
\end{proof}

\begin{remark}
	$(i)$ If $\psi$ is constant and $\phi(z)=tz$ where $-1\leq t <1,$ then $W(C_{\psi,\phi})$ is a line segment of $\mathbb{C}$ and thus $0 \notin $ $int$ $W(C_{\psi,\phi}).$\\
	$(ii)$ If $\psi$ is nonconstant with $\psi(0)=0$ and $\phi(z)=tz$ where $0<t<1,$ then by Lemma \ref{lem1} we have $0 \in$ $int$ $W(C_{\psi,\phi}).$\\
	$(iii)$ If $\psi$ is nonconstant with $\psi(0)\neq 0$ and $\phi(z)=tz$ where $0<t<1,$ then the following two cases are possible.\\
	 $(a)$ Let $\psi$ be such that $\psi(0)\neq 0$ but $\psi(w_0)=0$ for some $w_0\in \mathbb{D}.$ Then by applying Lemma \ref{lem1} we conclude that $0 \in$ $int$ $W(C_{\psi,\phi}).$\\
	 $(b)$ Now, we consider $\phi(z)=z/2$ and $\psi(z)=1+z/4.$ Let $f \in L_a^2(dA_{\alpha})$ with $\|f\|=1$ and $f(z)=\sum_{k=0}^{\infty}\hat{f}_kz^k$ be the power series expansion. Then we have
	 \begin{align*}
	    \langle  C_{\psi,\phi}f,f \rangle&=\sum_{k=0}^{\infty}\frac{k!\Gamma(\alpha+2)}{2^k\Gamma(k+\alpha+2)}|\hat{f}_k|^2+\frac{1}{4}\sum_{k=0}^{\infty}\frac{(k+1)!\Gamma(\alpha+2)}{2^k\Gamma(k+\alpha+3)}\hat{f}_{k}\bar{\hat{f}}_{k+1}\\
	    &=\gamma+\frac{1}{4}\delta,
	 \end{align*}
     where $\gamma=\sum_{k=0}^{\infty}\frac{k!\Gamma(\alpha+2)}{2^k\Gamma(k+\alpha+2)}|\hat{f}_k|^2$ and $\delta=\sum_{k=0}^{\infty}\frac{(k+1)!\Gamma(\alpha+2)}{2^k\Gamma(k+\alpha+3)}\hat{f}_{k}\bar{\hat{f}}_{k+1}.$\\
     Since $|\hat{f}_{k}\bar{\hat{f}}_{k+1}|\leq \frac{1}{2}\left(|\hat{f}_{k}|^2+|\hat{f}_{k+1}|^2\right)$ and a by simple computation we have 
     \begin{align*}
     	|\delta| \leq \frac{3}{2}\gamma  \,\,\,\,\mbox{as $\gamma>0.$}
     \end{align*}
     Therefore $\gamma+\frac{1}{4}\delta \neq 0$ and this implies that $0 \notin W(C_{\psi,\phi}).$
\end{remark}

Finally, we introduce some sufficient condition for the closedness of numerical range of compact weighted composition operators. The results follows from the previous results of this section and the well known result that if the numerical range of a compact operator contains $0$ then it is closed, see \cite[p. 115]{H_BOOK}.

\begin{cor}\label{c11}
	Let $C_{\psi,\phi}$ be a compact operator on  $L_a^2(dA_{\alpha}).$ Then $W(C_{\psi,\phi})$ is closed if one of the following condition holds:\\
	(i)~~$\phi$ is the identity map on $\mathbb{D}$ and $\psi$ has a zero on $\mathbb{D}.$\\
	(ii)~~$\psi \neq 0$ and either $\psi$ has a zero on $\mathbb{D}$ or $\phi$ is not one-to-one.\\
	(iii)~~$\phi(0)=0$ and $\phi$ is not of the form $\phi(z)=tz$ for $t \in \overline{\mathbb{D}}$.\\
	(iv)~~$\psi$ is nonconstant and $\phi(z)=tz,$ $-1\leq t\leq 0.$
\end{cor}

\begin{example}
	If we consider the example $\psi(z)=z^2$ and $\phi(z)=az^2$ with $|a|<1$ then $\phi$ is not of the form $\phi(z)=tz$ for $t \in \overline{\mathbb{D}},$ and $\phi(0)=0.$ Again from \cite[Cor. 1]{M_JFA_2005} it follows that $C_{\psi,\phi}$ is compact on  $L_a^2(dA_{\alpha}).$ Thus for this example $W(C_{\psi,\phi})$ is closed.
\end{example}

\section{Containment of circle or ellipse in the numerical range}\label{S3}

In the previous section we study about the containment of the origin in the interior of numerical range. Our next focus is to investigate the weighted composition operators for which the numerical range contains a circular disc or ellipse and accordingly we find the radius of the disc or lengths of the major and minor axis.

\begin{theorem}\label{TH1}
	Let	$C_{\psi,\phi} \in  \mathbb{B}\left({L_a^2(dA_{\alpha})}\right)$ be such that $\phi(0)=0$ and $\psi$ has a zero of order $m>0$ at the origin. If $\hat{\psi}_m$ denotes the $m$-th Taylor coefficient of $\psi,$ then $W(C_{\psi,\phi})$ contains the disc of radius $\frac{m!\Gamma(\alpha+2)}{\Gamma(m+\alpha+2)+m!\Gamma(\alpha+2)}|\hat{\psi}_m|$, centred at the origin.
\end{theorem}

\begin{proof}
	Let us consider $f(z)=\sqrt{\frac{\Gamma(m+\alpha+2)}{\Gamma(m+\alpha+2)+m!\Gamma(\alpha+2)}}(\lambda+z^m)$ for all $z \in \mathbb{D}$ with $|\lambda|=1.$ Then $f \in L_a^2(dA_{\alpha})$ with $\|f\|=1.$ Let $\sum_{k=1}^{\infty}\hat{\phi}_kz^k$ and $\sum_{k=m}^{\infty}\hat{\psi}_kz^k$ be the Taylor series of $\phi$ and $\psi,$ respectively. Then
	\begin{align*}
		f(\phi(z))&=\sqrt{\frac{\Gamma(m+\alpha+2)}{\Gamma(m+\alpha+2)+m!\Gamma(\alpha+2)}}
		\left( \lambda + \left(\sum_{k=1}^{\infty}\hat{\phi}_kz^k\right)^m\right)\\
		&=\sqrt{\frac{\Gamma(m+\alpha+2)}{\Gamma(m+\alpha+2)+m!\Gamma(\alpha+2)}}
		\left( \lambda + \hat{\phi}_1^mz^m+ \mbox{higher order terms of $z$} \right).
	\end{align*}
	Thus we have
	\begin{align*}
		& \langle C_{\psi,\phi}f,f \rangle\\
		&= \langle \psi(z)f(\phi(z)),f(z) \rangle \\
		& = \frac{\Gamma(m+\alpha+2)}{\Gamma(m+\alpha+2)+m!\Gamma(\alpha+2)} \left \langle  \left( \sum_{k=m}^{\infty}\hat{\psi}_kz^k \right)\left( \lambda + \hat{\phi}_1^mz^m+ \mbox{higher order terms of $z$} \right),\lambda+z^m\right\rangle \\
		&=\frac{m!\Gamma(\alpha+2)}{\Gamma(m+\alpha+2)+m!\Gamma(\alpha+2)}\lambda \hat{\psi}_m.
	\end{align*} 
	Since $\lambda$ is an arbitrary complex number with $|\lambda|=1$ so $W(C_{\psi,\phi})$ contains the disc of radius $\frac{m!\Gamma(\alpha+2)}{\Gamma(m+\alpha+2)+m!\Gamma(\alpha+2)}|\hat{\psi}_m|$ and centre at the origin.
\end{proof}

\begin{theorem}\label{TH2}
	Let	$C_{\psi,\phi} \in  \mathbb{B}\left({L_a^2(dA_{\alpha})}\right)$ be such that $\phi(z)=\lambda z$ with $\lambda \neq 0$ and $\psi(z)=\sum_{k=1}^{\infty}\hat{\psi}_kz^k.$ Then  for all $m\geq 2,$ $W(C_{\psi,\phi})$ contains the disc of radius $\frac{1}{2}\sqrt{\frac{m!\Gamma(\alpha+3)}{\Gamma(m+\alpha+2)}}|\lambda \hat{\psi}_{m-1}|$ and centre at the origin.
\end{theorem}

\begin{proof}
	For $m \geq 2$ let $M_m$ be the subspace of $L_a^2(dA_{\alpha})$ spanned by $e_1$ and $e_m.$ Now, we have
	\begin{align*}
		C_{\psi,\phi}e_1(z)=\lambda \sqrt{\alpha+2}\left( \sum_{k=1}^{\infty}\hat{\psi}_kz^{k+1} \right)
	\end{align*}
	and
	\begin{align*}
		C_{\psi,\phi}e_m(z)=\lambda^m \sqrt{\frac{\Gamma(m+\alpha+2)}{m!\Gamma(\alpha+2)}}\left( \sum_{k=1}^{\infty}\hat{\psi}_kz^{k+m} \right).
	\end{align*}
	Thus the matrix representation of the compression of $C_{\psi,\phi}$ to $M_m$ has the matrix representation
	\begin{align*} 
		\begin{pmatrix}
			0&0\\
			\lambda \sqrt{\frac{m!\Gamma(\alpha+3)}{\Gamma(m+\alpha+2)}}\hat{\psi}_{m-1} &0
		\end{pmatrix}.
	\end{align*}
	Therefore, the numerical range of the compression of $C_{\psi,\phi}$ to $M_m$ is a closed disc centred at the origin and radius $\frac{1}{2}\sqrt{\frac{m!\Gamma(\alpha+3)}{\Gamma(m+\alpha+2)}}|\lambda \hat{\psi}_{m-1}|,$ see \cite[Example 3]{GR_BOOK}. Thus $W(C_{\psi,\phi})$ contains the disc of radius $\frac{1}{2}\sqrt{\frac{m!\Gamma(\alpha+3)}{\Gamma(m+\alpha+2)}}|\lambda \hat{\psi}_{m-1}|$ and centre at the origin.
\end{proof}

\begin{remark}
	Here we remark that the case $m=0$ cannot be included in Theorem \ref{TH1} and Theorem \ref{TH2}. If we consider $\psi(z)=1+\epsilon z$ and $\phi(z)=-z/2$ with $\epsilon>0$ then it is easy to observe that the radius of the largest circle centred at the origin and contained in $W(C_{\psi,\phi})$ is at most $\epsilon.$	
\end{remark}

\begin{theorem}
	Let	$C_{\psi,\phi} \in  \mathbb{B}\left({L_a^2(dA_{\alpha})}\right)$ be such that $\phi(z)=e^{2\pi i/n} z$ and $\psi(z)=\sum_{k=0}^{\infty}\hat{\psi}_kz^k.$ If $m_1,m_2$ are two positive integers with $m_2>m_1$ and $\hat{\psi}_{nm_1}\hat{\psi}_{nm_2}\hat{\psi}_{n(m_1-m_2)}=0$ but all of the three terms $\hat{\psi}_{nm_1},\hat{\psi}_{nm_2}$ and $\hat{\psi}_{n(m_1-m_2)}$ are not equal to zero. Then $W(C_{\psi,\phi})$ contains the circle centred at $1$ with radius 
	$$\frac{1}{2}\sqrt{\frac{(nm_2)!\Gamma(\alpha+2)}{\Gamma(nm_2+\alpha+2)}\left(c|\hat{\psi}_{nm_1}|^2+|\hat{\psi}_{n(m_1-m_2)}|^2+\frac{1}{c}|\hat{\psi}_{nm_2}|^2\right)},$$ 
	where $c= \frac{(nm_1)!\Gamma(nm_2+\alpha+2)}{(nm_2)!\Gamma(nm_1+\alpha+2)}.$
\end{theorem}

\begin{proof}
	Let $M$ be the subspace of $L_a^2(dA_{\alpha})$ spanned by $e_0, e_{nm_1}$ and $e_{nm_2}.$ Then we have 
	\begin{align*}
		C_{\psi,\phi}e_0(z)=\sum_{k=0}^{\infty}\hat{\psi}_kz^{k},
	\end{align*}
	\begin{align*}
		C_{\psi,\phi}e_{nm_1}(z)=\sqrt{\frac{\Gamma(nm_1+\alpha+2)}{(nm_1)!\Gamma(\alpha+2)}}\sum_{k=0}^{\infty}\hat{\psi}_kz^{nm_1+k}
	\end{align*}
	and 
	\begin{align*}
		C_{\psi,\phi}e_{nm_2}(z)=\sqrt{\frac{\Gamma(nm_2+\alpha+2)}{(nm_2)!\Gamma(\alpha+2)}}\sum_{k=0}^{\infty}\hat{\psi}_kz^{nm_2+k}.
	\end{align*}
	Thus the matrix representation of the compression of $C_{\psi,\phi}$ to $M$ has the matrix representation
	\begin{align*} 
		\begin{pmatrix}
			1&0&0\\
			\sqrt{\frac{(nm_1)!\Gamma(\alpha+2)}{\Gamma(nm_1+\alpha+2)}}\hat{\psi}_{nm_1}&1&0\\
			\sqrt{\frac{(nm_2)!\Gamma(\alpha+2)}{\Gamma(nm_2+\alpha+2)}}\hat{\psi}_{nm_2}&
			\frac{\sqrt{\Gamma(nm_1+\alpha+2)\Gamma(\alpha+2)}(nm_2)!}{\sqrt{(nm_1)!}\Gamma(nm_2+\alpha+2)}\hat{\psi}_{n(m_1-m_2)}&1
		\end{pmatrix}.
	\end{align*}
	It follows from \cite[Th. 4.1]{KRS_LAA_1997} that the numerical range of the compression of $C_{\psi,\phi}$ to $M$ is the circle centred at $1$ with radius 
	$$\frac{1}{2}\sqrt{\frac{(nm_2)!\Gamma(\alpha+2)}{\Gamma(nm_2+\alpha+2)}\left(c|\hat{\psi}_{nm_1}|^2+|\hat{\psi}_{n(m_1-m_2)}|^2+\frac{1}{c}|\hat{\psi}_{nm_2}|^2\right)},$$ 
	where $c= \frac{(nm_1)!\Gamma(nm_2+\alpha+2)}{(nm_2)!\Gamma(nm_1+\alpha+2)}.$ As the numerical range of compression is contained in the numerical range of operator, this proves the desired result.
\end{proof}

\begin{theorem}
	Let	$C_{\psi,\phi} \in  \mathbb{B}\left({L_a^2(dA_{\alpha})}\right)$ be such that $\phi(z)=\lambda z$ with $\lambda=e^{2\pi i/n}$ and $\psi(z)=\sum_{k=0}^{\infty}\hat{\psi}_kz^k$ with $\hat{\psi}_{np+j} \neq 0$ for some $0<j<n.$ Then $W(C_{\psi,\phi})$ contains the ellipse with foci $\hat{\psi_0}$ and $\lambda^{np+j}\hat{\psi_0},$ and with major axis 
	$$\sqrt{|\hat{\psi_0}|^2|1-e^{2\pi ij/n}|^2+\frac{(np+j)!\Gamma(\alpha+2)}{\Gamma(np+j+\alpha+2)}|\hat{\psi}_{np+j}|^2}$$ 
	and minor axis $$\sqrt{\frac{(np+j)!\Gamma(\alpha+2)}{\Gamma(np+j+\alpha+2)}}|\hat{\psi}_{np+j}|.$$
\end{theorem}

\begin{proof}
	Let $M$ be the subspace of $L_a^2(dA_{\alpha})$ spanned by $e_0$ and $e_{np+j}.$  We have
	\begin{align*}
		C_{\psi,\phi}e_0(z)=\sum_{k=0}^{\infty}\hat{\psi}_kz^{k}
	\end{align*}
	and
	\begin{align*}
		C_{\psi,\phi}e_{np+j}(z)=\lambda^{np+j} \sqrt{\frac{\Gamma(np+j+\alpha+2)}{(np+j)!\Gamma(\alpha+2)}}\left( \sum_{k=0}^{\infty}\hat{\psi}_kz^{np+j+k} \right).
	\end{align*}
	Thus the matrix representation of the compression of $C_{\psi,\phi}$ to $M$ has the matrix representation
	\begin{align*} 
		\begin{pmatrix}
			\hat{\psi_0}&0\\
			\sqrt{\frac{(np+j)!\Gamma(\alpha+2)}{\Gamma(np+j+\alpha+2)}}\hat{\psi}_{np+j} &\lambda^{np+j}\hat{\psi_0}
		\end{pmatrix}.
	\end{align*}
	Since $0<j<n$ so $\lambda^{np+j} \neq 1$ and hence the the numerical range of the compression of $C_{\psi,\phi}$ to $M$ is the ellipse with foci $\hat{\psi_0}$ and $\lambda^{np+j}\hat{\psi_0},$ and with major axis 
	$\sqrt{|\hat{\psi_0}|^2|1-e^{2\pi ij/n}|^2+\frac{(np+j)!\Gamma(\alpha+2)}{\Gamma(np+j+\alpha+2)}|\hat{\psi}_{np+j}|^2}$ and minor axis $\sqrt{\frac{(np+j)!\Gamma(\alpha+2)}{\Gamma(np+j+\alpha+2)}}|\hat{\psi}_{np+j}|,$ see \cite[Example 3]{GR_BOOK}. The desired result follows from the fact that the numerical range of compression is contained in the numerical range of operator.
\end{proof}

\begin{theorem}
	Let	$C_{\psi,\phi} \in  \mathbb{B}\left({L_a^2(dA_{\alpha})}\right)$ be such that $\phi(z)=e^{2\pi i\theta}z$ and $\psi(z)=\sum_{k=0}^{\infty}\hat{\psi}_kz^k,$ where $\theta$ is irrational. If $n \geq 0$ and $m>0$ then $W(C_{\psi,\phi})$ contains the ellipse with foci at $e^{2\pi in\theta }$ and $e^{2\pi i(n+m)\theta },$ and with major axis 
	$$\sqrt{|e^{2\pi in\theta }-e^{2\pi i(n+m)\theta }|^2+\frac{(n+m)!\Gamma(n+\alpha+2)}{n!\Gamma(n+m+\alpha+2)}|\hat{\psi}_{m}|^2}$$
	and minor axis $\sqrt{\frac{(n+m)!\Gamma(n+\alpha+2)}{n!\Gamma(n+m+\alpha+2)}}|\hat{\psi}_{m}|.$
\end{theorem}

\begin{proof}
	Let $M$ be the subspace of $L_a^2(dA_{\alpha})$ spanned by $e_n$ and $e_{n+m}.$ Then we have 
	\begin{align*}
		C_{\psi,\phi}e_n(z)=e^{2\pi in\theta }\sqrt{\frac{\Gamma(n+\alpha+2)}{n!\Gamma(\alpha+2)}}\sum_{k=0}^{\infty}\hat{\psi}_kz^{k+n}
	\end{align*}
	and
	\begin{align*}
		C_{\psi,\phi}e_{n+m}(z)=e^{2\pi i(n+m)\theta }\sqrt{\frac{\Gamma(n+m+\alpha+2)}{(n+m)!\Gamma(\alpha+2)}}\sum_{k=0}^{\infty}\hat{\psi}_kz^{k+n+m}.
	\end{align*}
	Thus the matrix representation of the compression of $C_{\psi,\phi}$ to $M$ has the matrix representation
	\begin{align*} 
		\begin{pmatrix}
			e^{2\pi in\theta}&0\\
			e^{2\pi i(n+m)\theta }\sqrt{\frac{(n+m)!\Gamma(n+\alpha+2)}{n!\Gamma(n+m+\alpha+2)}}\hat{\psi}_{m}&e^{2\pi i(n+m)\theta}
		\end{pmatrix}.
	\end{align*}
	Therefore, the numerical range of the compression of $C_{\psi,\phi}$ to $M$ is the ellipse with foci $e^{2\pi in\theta }$ and $e^{2\pi i(n+m)\theta },$ and with major axis 
	$$\sqrt{|e^{2\pi in\theta }-e^{2\pi i(n+m)\theta }|^2+\frac{(n+m)!\Gamma(n+\alpha+2)}{n!\Gamma(n+m+\alpha+2)}|\hat{\psi}_{m}|^2}$$
	and minor axis $\sqrt{\frac{(n+m)!\Gamma(n+\alpha+2)}{n!\Gamma(n+m+\alpha+2)}}|\hat{\psi}_{m}|,$ see \cite[Example 3]{GR_BOOK}. The desired result follows as the numerical range of compression is contained in the numerical range of operator.
\end{proof}

We end with the conclusion that the results of this section may be useful to estimate the lower bounds of numerical radius of some classes of weighted composition operators acting on $L_a^2(dA_{\alpha}).$

\textit{Acknowledgements.} Mr. Anirban Sen would like to thank CSIR, Govt. of India, for the financial support in the form of Senior Research Fellowship under the mentorship of
Prof. Kallol Paul. Mr. Subhadip Halder would like to thank UGC, Govt. of India for the financial support in the form of Junior Research Fellowship under the mentorship of Dr. Riddhick Birbonshi.

\bibliographystyle{amsplain}

\end{document}